\documentclass[12pt]{article}

\usepackage{amsmath,amsthm,amsfonts,amssymb,amscd,cite}
\usepackage{graphicx}

\allowdisplaybreaks[4]

\newtheorem {lemma}{Lemma}
\newtheorem {theorem} {Theorem}

\newtheorem {corollary}{Corollary}

\allowdisplaybreaks [4]
\usepackage{fullpage}

\begin{document}

\title{The Alon-Tarsi number of $K_{3,3}$-minor-free graphs}

\author{Leyou Xu\footnote{Email: leyouxu@m.scnu.edu.cn}, Bo Zhou\footnote{Email: zhoubo@m.scnu.edu.cn}\\ School of Mathematical Sciences, South China Normal University\\ Guangzhou 510631, P.R. China}

\date{}
\maketitle

\begin{abstract}
The well known Wagner's theorem states that a graph is a planar graph if and only if it is $K_5$-minor-free and $K_{3,3}$-minor-free.
Denote by $AT(G)$ the Alon-Tarsi number of a graph $G$.
We show that for any $K_{3,3}$-minor-free graph $G$,  $AT(G)\le 5$,  there exists a matching $M$ and a forest $F$ such that $AT(G-M)\le 4$ and $AT(G-E(F))\le 3$, extending  the result
on the Alon-Tarsi number of $K_5$-minor-free  graphs due to Abe,  Kim and  Ozeki. \\ \\
{\bf Keywords: }  Alon-Tarsi number, $K_{3,3}$-minor-free graph, planar graph\\ \\
{\bf AMS Classifications:} 05C15
\end{abstract}

\section{Introduction}

We consider simple, finite and undirected graphs.  For a graph $G$, denote by  $V(G)$ the vertex set and $E(G)$ the edge set of $G$. A graph $H$ is a minor of a connected graph $G$ if we may obtain $H$ from $G$ by repeatedly deleting vertices and edges and contracting edges. A graph $G$ is $H$-minor-free if $H$ is not a minor of $G$. For $\emptyset\ne V_1\subseteq V(G)$, $G[V_1]$ denotes the subgraph of $G$ induced by $V_1$.
Denote by $K_n$  the complete graph on $n$ vertices and $K_{s,t}$ the complete bipartite graph on  $s$ and $t$ vertices in its color classes, respectively. A clique $S$ of a graph $G$ is a nonempty subset of $V(G)$ such that $G[S]$ is complete.

Let $G$ be a graph and $D$ an orientation of $G$. For $v\in V(G)$,  denote by $d_D^+(v)$ ($d_D^-(v)$, respectively) the out-degree (in-degree, respectively) of $v$ in $D$. The maximum out-degree of $D$ is denoted by $\Delta^+(D)$.
An Eulerian subdigraph (or circulation) $H$ of $D$ is a spanning subdigraph of $D$ such that $d_H^+(v)=d_H^-(v)$ for each vertex $v\in V(G)$.
Denote by $EE(D)$ ($OE(D)$, respectively) the set of all Eulerian subdigraphs of $D$ with even (odd, respectively) number of edges, respectively.  We say that $D$ is acyclic if $D$ does not contain any directed cycle.
Use  the Combinatorial Nullstellensatz, Alon and
Tarsi \cite{AT}  obtained a remarkable relationship between a special orientation of a graph and a certain graph polynomial, which  is known as Alon-Tarsi Theorem.
The Alon-Tarsi number $AT(G)$ of a graph $G$, introduced by Jensen and Toft \cite{JT}, is the minimum integer $k$ such that there exists an orientation $D$ of $G$ with $|EE(D)|\ne |OE(D)|$ and $\Delta^+(D)<k$.

A  coloring of a graph $G$ is a mapping $c: V(G)\rightarrow \mathbb{R}$.
For $d\ge 0$, a $d$-defective coloring of $G$ is a coloring  such that each color class induces a subgraph of maximum degree at most $d$. Particularly, a $0$-defective coloring is said to be proper.
%A coloring $c$ of  $G$ is proper if $c(u)\ne c(v)$ whenever $uv\in E(G)$.
A graph is $k$-colorable if there exists a proper coloring $c$ with $|c(V(G))|\le k$.
The chromatic number $\chi(G)$ of  $G$ is the least $k$ such that $G$ is $k$-colorable.

A list assignment of $G$ is a function $L$ assigning to every
vertex $v\in V(G)$ a set $L(v)\subset \mathbb{R}$.
Given a list assignment $L$, a $d$-defective-$L$-coloring of $G$ is a $d$-defective coloring $c$ such that $c(v)\in L(v)$ for each vertex $v\in V(G)$. %A $0$-defective-$L$-coloring is called an $L$-coloring.
A graph $G$ is $d$-defective $k$-choosable if there exists an $d$-defective-$L$-coloring of $G$ for each assignment $L$ with $|L(v)|\ge k$ for every $v\in V(G)$. Particularly, we say $G$ is $k$-choosable if $G$ is $0$-defective $k$-choosable. The list chromatic number $\chi_\ell(G)$ of  $G$ is the least integer $k$ such that $G$ is $k$-choosable.

By Alon-Tarsi Theorem  \cite{AT}, $\chi(G)\le \chi_\ell(G)\le AT(G)$ for any $G$.

The famous Four-Color Theorem states that for any planar graph $G$, $\chi(G)\le 4$.
Vizing \cite{Viz} asked whether every planar graph is $5$-choosable.
Erd\H{o}s, Rubin and Taylor \cite{ERT} conjectured
that every planar graph is $5$-choosable but not necessarily $4$-choosable.
Thomassen \cite{T} confirmed the former and
Voigt \cite{V} confirmed the latter. %showed that there are planar graphs which are not $4$-choosable.
%So $5$ is a tight upper bound for the list chromatic number of a planar graph.
Cushing and Kierstead \cite{CK} proved  that every planar graph is $1$-defective $4$-choosable.
Let $G$ be a planar graph.
Zhu  \cite{Zhu} improved Thomassen's upper bound by proving that $AT(G)\le 5$, solving a natural open problem \cite{He}. Grytczuk and Zhu \cite{GZ} showed that there exists a matching $M$ of $G$ such that $AT(G-M)\le 4$. Furthermore, Kim,  Kim and  Zhu \cite{KKZ} showed that there exists a forest $F$ in $G$ such that $AT(G-E(F))\le 3$.

By the well known Wagner's theorem \cite{W}, a graph is a planar graph if and only if it is $K_5$-minor-free and $K_{3,3}$-minor-free.
Recently, Abe,  Kim and  Ozeki \cite{AKO} extended the above results on the Alon-Tarsi number from  planar graphs  to $K_5$-minor-free graphs.

\begin{theorem}\label{minor} \cite[Theorem 1.6, Corollary 1.8]{AKO}
Let $G$ be a $K_5$-minor-free graph. The following statements are true.

\begin{enumerate}
\item[(i)]  $AT(G)\le 5$.

\item[(ii)]  There exists a matching $M$ of $G$ such that $AT(G-M)\le 4$.

\item[(iii)]  There exists a forest $F$ in $G$ such that $AT(G-E(F))\le 3$.
\end{enumerate}
\end{theorem}

In this note, we extend these results in another direction to  $K_{3,3}$-minor-free graphs.

\begin{theorem}\label{kat}
Let $G$ be a $K_{3,3}$-minor-free graph. Then the following results hold.
\begin{enumerate}
\item[(i)] $AT(G)\le 5$.
\item[(ii)]  There exists a matching $M$ of $G$ such that $AT(G-M)\le 4$.
\item[(iii)]   There exists a forest $F$ of $G$ such that $AT(G-E(F))\le 3$.
\end{enumerate}
\end{theorem}

As mentioned above, there exist planar graphs which are not $4$-choosable \cite{V}, which are surely $K_{3,3}$-minor-free, so the bound in Theorem \ref{kat} (i) is tight.
Theorem \ref{kat} (ii) implies that
 a $K_{3,3}$-minor-free graph is $1$-defective $4$-choosable.

\begin{corollary}\label{color} \cite{Wo}
Let $G$ be a $K_{3,3}$-minor-free graph. Then $\chi_\ell(G)\le 5$.
\end{corollary}

The study of list coloring of $K_{s,t}$-minor-free graphs has received much attention \cite{Wo, St}.    Steiner \cite{St} proved that for every pair of constants $\epsilon>0$ and $C>1$,  there exists a positive integer $N=N(\epsilon,C)$ such that for all integers $s$ and $t$ satisfying $N\le s\le t\le Cs$, there exists a  $K_{s,t}$-minor-free graph $G$ such that $\chi_\ell(G)>(1-\epsilon)(2s+t)$.
For any such graph $G$, $AT(G)>(1-\epsilon)(2s+t)$. So, in general,  it is impossible that $AT(G)\le s+t-1$ for a
 $K_{s,t}$-minor-free graph $G$.

\section{Proof of Theorem \ref{kat}}

We  need the following important result due to Zhu and his coauthors \cite{Zhu,GZ,KKZ}. In a plane graph every face  is bounded by a closed walk (not necessarily a cycle), which is called
a boundary walk. An orientation $D$ of a graph $G$ is said to be an AT-orientation if $|EE(D)|\ne |OE(D)|$.

\begin{lemma}\label{atp}
Let $G$ be a nontrivial plane graph with a boundary walk $v_1\dots v_m$ of the infinite face. The following statements are true.
\begin{enumerate}
\item[(a)] \cite{Zhu} There exists an AT-orientation $D$ of $G$ such that $d_D^+(v_1)=0$, $d_D^+(v_2)=1$, $d_D^+(v_i)\le 2$ for each $i=3,\dots,m$ and $\Delta^+(D)\le 4$.
\item[(b)] \cite{GZ} There exists a matching $M$ of $G$ and an AT-orientation $D$ of $G-M$ such that $d_D^+(v_1)=d_D^+(v_2)=0$, $d_D^+(v_i)\le 2-d_M(v_i)$ for each $i=3,\dots,m$ and $\Delta^+(D)\le 3$.
\item[(c)] \cite{KKZ} There exists a forest $F$ of $G$ and an acyclic orientation $D$ of $G-E(F)$ such that $d_D^+(v_1)=d_D^+(v_2)=0$,  $d_D^+(v_i)=1$ for each $i=3,\dots,m$ and $\Delta^+(D)\le 2$.
\end{enumerate}
\end{lemma}

Let $G_1$ and $G_2$ be two vertex disjoint graphs.
Suppose that $X_i\subset  V(G_i)$  is a clique of $G_i$ for $i=1,2$ with $|X_1|=|X_2|=k$.
Let $f: X_1\rightarrow X_2$ be a bijection.
A graph $G$ obtained from $G_1$ and $G_2$ by identifying $x$ and $f(x)$ for every $x\in X_1$
 and possibly deleting some edges of the clique %so that $G[X_1]=G[X_2]\subseteq K_k$
is called a $k$-clique-sum of $G_1$ and $G_2$. Evidently, a $0$-sum of $G_1$ and $G_2$ is $G_1\cup G_2$.

\begin{lemma}\label{k33}\cite{W}
A graph $G$ is $K_{3,3}$-minor-free if and  only  if $G$ is a planar graph or $K_5$, or $G$ can be obtained from planar graphs and $K_5$ by $0$-, $1$-, and $2$-sums.
\end{lemma}

\begin{lemma}\label{ksum}
Let $G$ be a graph obtained by the $k$-clique-sum of $G_1$ and $G_2$, where $k\ge 1$. Let $K$ be their common clique with $K=\{v_1,\dots,v_k\}$. Let $G_i'=G_i\cap G$ for $i=1,2$.
Suppose that $G_1'$ has an AT-orientation $D_1'$ with $\Delta^+(D_1')\le \ell$
and that $G_2$ has an AT-orientation $D_2$ with $\Delta^+(D_2)\le \ell$ and $d^+_{D_2}(v_i)=i-1$ for $i=1,\dots,k$. Then $G$ has an AT-orientation $D$ such that $\Delta^+(D)\le \ell$ and $d_D^+(v)=d_{D_1'}^+(v)$ for each $v\in V(G_1)$.
\end{lemma}

\begin{proof}
Let $D_2'$ be restriction of $D_2$ on $G_2-E(G[K])$. Then $D=D_1'\cup D_2'$ is an orientation of $G$.
As $d_{D_2}^+(v_i)=i-1$ for $i=1,\dots,k$ and $D_2[K]$ is a tournament, $D_2[K]$ is transitive and there is no arc from $K$ to $V(G_2)\setminus K$ in $D_2$.
Then $d_{D_2'}^+(u)=0$ for $u\in K$.
So $d_D^+(v)=d_{D_1'}^+(v)$ for each $v\in V(G_1)$. Note that $d_D^+(v)=d_{D_2'}^+(v)=d_{D_2}^+(v)$ for each $v\in V(G_2)\setminus K$. Thus $\Delta^+(D)=\max\{\Delta^+(D_1'),\Delta^+(D_2')\}\le \ell$. We are left to show that $D$ is an AT-orientation of $G$.
	
Let $H$ be an Eulerian subdigraph of $D$ and $H_1=H[V(G_1)]$.

\noindent
{\bf Claim 1.}  $H$ contains no arcs  from $V(G_2)\setminus K$ to $K$.
	
\begin{proof}
If there exists some vertex $v\in K$ and $w\in V(G_2)\setminus K$ such that $(w,v)$ is an arc of $H$, then $d_{H_1}^-(v)<d_{H}^{-}(v)$.
As there is no arcs from $K$ to $V(G_2)\setminus K$ in $D_2$, we have $d_{H_1}^+(u)=d_H^+(u)$ for each $u\in K$. Thus $d_{H_1}^+(u)=d_H^+(u)$ for each $u\in V(G_1)$, so
\[
\sum_{u\in V(H_1)}d_{H}^+(u)=\sum_{u\in V(H_1)}d_{H_1}^+(u)
=|A(H_1)|=\sum_{u\in V(H_1)}d_{H_1}^-(u)
<\sum_{u\in V(H_1)}d_{H}^-(u),
\]
a contradiction.
\end{proof}
	
By Claim 1, $H$ contains no arcs either from $K$ to $V(G_2)\setminus K$ or from $V(G_2)\setminus K$ to $K$.
Therefore, $H$ has an edge-disjoint decomposition $H=H_1\cup H_2$, where $H_1$ and $H_2$ are Eulerian subdigraphs of $D_1'$ and $D_2'$, respectively.

If $H\in EE(D)$, then either $H_1\in EE(D_1')$ and $H_2\in EE(D_2')$ or $H_1\in OE(D_1')$ and $H_2\in OE(D_2')$.
%If $H\in OE(D)$, then either $H_1\in EE(D_1')$ and $H_2\in OE(D_2')$ or
%$H_1\in OE(D_1')$ and $H_2\in EE(D_2')$.
On the other hand, $H_1'\cup H_2'\in EE(D)$ for any $H_i'\in EE(D_i')$ with $i=1,2$ or $H_i'\in OE(D_1')$ with $i=1,2$.
%For any $H_1'\in EE(D_1')$ and $H_2'\in OE(D_2')$, or  $H_1'\in OE(D_1')$ and $H_2'\in EE(D_2')$, $H_1'\cup H_2'\in OE(D)$.
Thus, there is a bijection between $EE(D)$ and $(EE(D_1')\times EE(D_2'))\cup (OE(D_1')\times OE(D_2'))$. Similarly, there is a bijection between $OE(D)$ and $(EE(D_1')\times OE(D_2'))\cup (OE(D_1')\times EE(D_2'))$.
Note that any Eulerian subdigraph of $D_2$ contains no arc incident to vertices in $K$ as $d_{D_2}^+(v_i)=i-1$ for $i=1,\dots,k$. Thus $EE(D_2')=EE(D_2)$ and $OE(D_2')=OE(D_2)$.
Recall that $D_1'$ and $D_2$ are AT-orientation. Thus $|EE(D_1')|-|OE(D_1')|\ne 0$ and $|EE(D_2')|-|OE(D_2')|\ne 0$.
It hence follows that
\begin{align*}
|EE(D)|-|OE(D)|&=\left(|EE(D_1')|\times |EE(D_2')|+ |OE(D_1')|\times |OE(D_2')|\right)\\
&\quad -\left(|EE(D_1')|\times |OE(D_2')|+ |OE(D_1')|\times |EE(D_2')|\right)\\
&=\left(|EE(D_1')|-|OE(D_1')| \right)\left(|EE(D_2')|-|OE(D_2')| \right)\\
&\ne 0,
\end{align*}
which implies that $D$ is an AT-orientation of $G$.
\end{proof}

Now we are ready to prove Theorem \ref{kat}.

\begin{proof}[Proof of Theorem \ref{kat}]
It suffices to show that for each $uv\in E(G)$, we have
\begin{enumerate}
\item[(a)] There exists an AT-orientation $D$ of $G$ such that $\Delta^+(D)\le 4$, $d_D^+(u)=0$ and $d_D^+(v)=1$.
\item[(b)] There exists a matching $M$ of $G$ and an AT-orientation $D$ of $G-M$ such that $\Delta^+(D)\le 3$ and $d_D^+(u)=d_D^+(v)=0$.
\item[(c)] There exists a forest $F$ of $G$ and an acyclic orientation $D$ of $G-E(F)$ such that
$\Delta^+(D)\le 2$ and $d_D^+(u)=d_D^+(v)=0$.
\end{enumerate}

If $G$ is a planar graph, then we may assume that $G$ is an plane graph so that edge $uv$ lies on the boundary of the infinite face,  so Item
(a) ((b), (c), respectively)  follows from (a) ((b), (c), respectively)  of Lemma \ref{atp}.

Suppose that $G\cong K_5$. Let $V(G)=\{v_1,\dots,v_5\}$ with $u=v_1,v=v_2$.

Let $D$ be the orientation of $G$ such  that $(v_i,v_j)$ is an arc of $D$ if and only if $i>j$. It is  obvious that $d_D^+(v_1)=0$, $d_D^+(v_2)=1$ and $\Delta^+(D)=4$. As $D$ is an acyclic orientation, it is an AT-orientation. So (a) follows.

Let $M=\{v_1v_2,v_4v_5 \}$ and let $D$ be the orientation of $G-M$ such that  $(v_i,v_j)$ is an arc of $D$ if and only if $i>j$. Then  $d_D^+(v_1)=d_D^+(v_2)=0$ and $\Delta^+(D)=3$. As $D$ is an acyclic orientation, it is an AT-orientation. So (b) follows.

Let $F$ be a forest with $E(F)=\{v_1v_2,v_3v_5,v_2v_4,v_4v_5\}$ and let $D$ be the orientation of $G-E(F)$ such that $(v_i,v_j)$ is an arc of $D$ if and only if $i>j$. Then $d_D^+(v_1)=d_D^+(v_2)=0$ and $\Delta^+(D)=2$. It is easy to see that $D$ is acyclic, so (c) follows.

Suppose that $G$ is not planar and $G\ncong K_5$.
Suppose by contradiction that %Item (a) ((b), (c), respectively) is not true. Let
 $G$ is a minimum counterexample with respect to the order. Then  $G$ is connected by the minimality of $G$.
As $G$ is $K_{3,3}$-minor-free, we have by Lemma \ref{k33} that there exists two $K_{3,3}$-minor-free graphs $G_1$ and $G_2$ such that $G$ is a $k$-clique-sum of $G_1$ and $G_2$ with  $k=1,2$.
Let $K$ be the common clique of $G_1$ and $G_2$ with $K=\{x\}$ if $k=1$ and $K=\{x,y\}$ if $k=2$.
Let $G_i'=G\cap G_i$  for $i=1,2$.
Let $uv\in  E(G_1')$. By the minimality of $G$, we have
\begin{enumerate}
\item[(a1)] There exists an AT-orientation $D_1$ of $G_1'$ such that $\Delta^+(D_1)\le 4$, $d_{D_1}^+(u)=0$ and $d_{D_1}^+(v)=1$.
\item[(b1)] There exists a matching $M_1$ of $G_1'$ and an AT-orientation $D_1$ of $G_1'-M_1$ such that $\Delta^+(D_1)\le 3$ and  $d_{D_1}^+(u)=d_{D_1}^+(v)=0$.
\item[(c1)] There exists a forest $F_1$ of $G_1'$ and an acyclic orientation $D_1$ of $G_1'-E(F_1)$ such that $\Delta^+(D_1)\le 2$ and $d_{D_1}^+(u)=d_{D_1}^+(v)=0$.
\end{enumerate}

Firstly, we show (a). By the minimality of $G$, there exists an AT-orientation $D_2$ of $G_2$ such that $\Delta^+(D_2)\le 4$, $d_{D_2}^+(x)=0$ and if $k=2$, then  $d_{D_2}^+(y)=1$. So, by (a1) and Lemma \ref{ksum}, $G$ has an AT-orientation $D$ such that $\Delta^+(D)\le 4$, $d_D^+(u)=0$, and $d_D^+(v)=1$,  contradicting the choice of $G$.

Secondly, we show (b).
Denote by $y$ a neighbor of $x$ in $G_2$ if $k=1$.
By the minimality of $G$, there exists a matching $M_2$ of $G_2$ and an AT-orientation $D_2$ of $G_2-M_2$ such that $\Delta^+(D_2)\le 3$, $d_{D_2}^+(x)=0$ and $d_{D_2}^+(y)=0$. Then $xy\in M_2$
and
neither $x$ nor $y$ is $M_2\setminus \{xy\}$-saturated.
Let $D_2'$ be an orientation of $G_2-(M_2\setminus \{xy\})$ with $D_2'=D_2\cup \{(y,x)\}$. Then $d_{D_2'}^+(y)=1$. So $D_2'$ is an AT-orientation with $\Delta^+(D_2')\le 3$, $d_{D_2'}^+(x)=0$ and $d_{D_2'}^+(y)=1$.
Let $M=M_1\cup (M_2\setminus \{xy\})$. Then $M$ is a matching of $G$. By (b1) and Lemma \ref{ksum}, $G-M$ has an AT-orientation $D$ such that $\Delta^+(D)\le 3$ and $d_D^+(u)=d_D^+(v)=0$,  contradicting  the choice of $G$.

Now we show (c).
Denote by $y$ a neighbor of $x$ in $G_2$ if $k=1$.
By the minimality of $G$, there exists a forest $F_2$ of $G_2$ and an acyclic orientation of $G_2-E(F_2)$ such that  $\Delta^+(D_2)\le 2$, $d_{D_2}^+(x)=0$ and $d_{D_2}^+(y)=0$. Then $xy\in E(F_2)$. Let $D_2'$ be an orientation of $G_2-(E(F_2)\setminus \{xy\})$ with $D_2'=D_2\cup \{(y,x)\}$. Then $d_{D_2'}^+(y)=1$. So $D_2'$ is an acyclic orientation with $\Delta^+(D_2')\le 2$, $d_{D_2'}^+(x)=0$ and $d_{D_2'}^+(y)=1$. Let $F=F_1\cup (F_2-xy)$. Then $F$ is a forest of $G$. By (c1) and Lemma \ref{ksum}, $G-E(F)$  has an acyclic orientation $D$ such that $\Delta^+(D)\le 2$ and $d_D^+(u)=d_D^+(v)=0$,  contradicting  the choice of $G$.
\end{proof}

A graph $G$ is $k$-degenerate if each subgraph of $G$ contains a vertex of degree at most $k$,
equivalently, $G$ has an acyclic orientation $D$ with $\Delta^+(D)\le k$.

\begin{corollary}
Let $G$ be a $K_{3,3}$-minor-free graph. Then there exists a forest $F$ such that $G-E(F)$ is $2$-degenerate.
\end{corollary}
\begin{proof} By the proof of Theorem \ref{kat}, there exists a forest $F$ and an acyclic AT-orientation $D$ such that $\Delta^+(D)\le 2$. We may label the vertices in $G$ such that $(v_i,v_j)$ is an arc of $D$ if and only if $i>j$. Let $H\subseteq G-E(F)$ and $H'$ be the orientation of $H$ on $D$. Let $t=\max\{v_i:v_i\in V(H) \}$. Then $d_{H'}^+(v_t)\le 2$ and $d_{H'}^-(v_t)=0$ and so the degree of $v_t$ is at most $2$ in $H$. Therefore, the result follows.
\end{proof}

\noindent
{\bf Remark 1.}  Theorem \ref{minor} may be extended to $K_5^{\bot}$-minor-free graphs, where
 $K_5^{\bot}$ is the graph obtained from $K_4$ by adding two adjacent vertices of degree three with no common neighbors.
From \cite[Lemmas 2.4 and 4.1]{DL}, a graph $G$ is $K_5^{\bot}$-minor free if and only if $G$ a $K_5$-minor-free graph or  $K_5$, 
 or $G$ can be obtained from $K_5$-minor-free graphs and $K_5$  by $0$-, $1$-, and $2$-sums.
The result follows from a similar argument as in the proof in Theorem \ref{kat} (that $K_5$ satisfies (a), (b) and (c) there) and \cite[Lemma 3.1]{AKO}.

\noindent
{\bf Remark 2.} Theorem \ref{kat} may be extended to $(K_{3,3}+e)$-minor-free graphs. By
\cite[Lemmas 2.4 and 3.7]{DL} \label{lk33e}
a graph $G$ is $(K_{3,3}+e)$-minor free graph if and only if $G$ is a $K_{3,3}$-minor-graph or $K_{3,3}$, or $G$ can be obtained from $K_{3,3}$-minor-free graphs and $K_{3,3}$ by $0$-, $1$- and $2$-sums.
So, by similar argument as in the  proof of Theorem \ref{kat}, we only need to check the validity of  (a)--(c) there if  $G=K_{3,3}$. This is verified as follows:  Let $\{u_1,u_2,u_3\}\cup \{v_1,v_2,v_3\}$ be the bipartition of $G$ with $u=u_1$ and $v=v_1$.

(a) Let $D$ be the orientation of $G$ with arc set
$\{(v_i,u_j):i=2,3,j=1,2,3 \}\cup \{(v_1,u_1),(u_2,v_1),(u_3,v_1)\}$. Then $d_D^+(u_1)=0$, $d_D^+(v_1)=1$ and $\Delta^+(D)=3$. As $D$ is an acyclic orientation, it is an AT-orientation.

(b) Let $M=\{v_iu_i:i=1,2,3\}$.
Let $D$ be the orientation of $G-M$ with arc set $\{(v_2,u_1),(v_3,u_1),(u_2,v_1),(u_2,v_3),(u_3,v_1),(u_3,v_2)\}$. Then $d_D^+(u_1)=d_D^+(v_1)=0$ and $\Delta^+(D)=2$. As $D$ is an acyclic orientation, it is an AT-orientation.

(c) Let $F$ be a forest with $E(F)=\{u_1v_1,u_1v_2,u_2v_2,u_2v_3,u_3v_3\}$ and let $D$ be the orientation of $G-E(F)$ with arc set $\{(v_3,u_1),(u_2,v_1),(u_3,v_1),(u_3,v_2)\}$.
Then $d_D^+(u_1)=d_D^+(v_1)=0$ and $\Delta^+(D)=2$. As $D$ is an acyclic orientation, it is an AT-orientation.

\section{Future research direction}

As mentioned above, it is impossible that $AT(G)\le s+t-1$ for a
 $K_{s,t}$-minor-free graph $G$ in general. It is of interest to find a tight upper bound of the Alon-Tarsi number for any  $K_{s,t}$-minor-free graph.  Steiner \cite{St} proposed an open problem:
Is it true that for all integers $1\le s\le t$, every $K_{s,t}$-minor-free graph $G$
satisfies  $\chi_\ell(G)\le 2s+t$? Similarly, one may ask whether it is true that for all integers $1\le s\le t$, every $K_{s,t}$-minor-free graph $G$
satisfies  $AT(G)\le 2s+t$.

\bigskip

\noindent {\bf Acknowledgement.}
%The authors thank Prof. Xuding Zhu for his encouragement.
This work was supported by the National Natural Science Foundation of China (No.~12071158).

\end{document}